\documentclass{amsart}
\usepackage{mathrsfs}
\usepackage{graphicx}
\usepackage{amsmath}
\usepackage{amscd}
\usepackage{amssymb}

\input xy
\xyoption{all}

\newcommand{\pd}{\partial}
\newcommand{\bC}{{\mathbb C}}

\newcommand{\bP}{{\mathbb P}}

\newcommand{\bR}{{\mathbb R}}
\newcommand{\bZ}{{\mathbb Z}}

\newcommand{\cO}{{\mathcal O}}

\newcommand{\half}{\frac{1}{2}}

\newtheorem{theorem}{Theorem}[section]
\newtheorem{theorem/definition}{Theorem/Definition}[section]

\newtheorem{prop}{Proposition}[section]

\theoremstyle{remark}

\theoremstyle{definition}

\newcommand{\be}{\begin{equation}}
\newcommand{\ee}{\end{equation}}
\newcommand{\bea}{\begin{eqnarray}}
\newcommand{\ben}{\begin{eqnarray*}}
\newcommand{\een}{\end{eqnarray*}}
\newcommand{\eea}{\end{eqnarray}}
\newcommand{\bet}{\begin{equation}
\begin{split}}
\newcommand{\eet}{\end{split}
\end{equation}}

\usepackage{color}

\definecolor{yellow}{rgb}{1,1,0}
\definecolor{orange}{rgb}{1,.7,0}
\definecolor{red}{rgb}{1,0,0} \definecolor{green}{rgb}{0,1,1}
\definecolor{white}{rgb}{1,1,1}

\definecolor{A}{rgb}{.75,1,.75}

\newcommand\Red[1]{\color{red}{#1}\color{black}\,}

\begin{document}

\title
{Hessian Geometry and Phase Change of Gibbons-Hawking Metrics}

\author{Jian Zhou}
\address{Department of Mathematical Sciences\\
Tsinghua University\\Beijing, 100084, China}
\email{jzhou@math.tsinghua.edu.cn}

\begin{abstract}
We study the Hessian geometry of toric Gibbons-Hawking metrics and their phase change phenomena
via the images of their moment maps.
\end{abstract}
\maketitle


\section{Introduction}

This is a sequel to \cite{Zhou-Kepler} in which techniques developed in string theory
were applied to study the Kepler problem in classical gravity.
We will apply the same techniques in this paper to study
gravitational instantons in Euclidean gravity in dimension four of type $A_{n-1}$,
i.e., the Gibbons-Hawking metrics.

Since the advent of AdS/CFT correspondence \cite{Mal, Wit} in string theory,
which are based supergravity solutions,
there have appeared new constructions of Sasaki-Einstein metrics \cite{GMSW}
which were soon realized to be toric \cite{Mar-Spa}.
Symplectic techniques developed for toric K\"ahler metrics on compact
manifolds was generalized to the metric cones of the toric Sasaki-Einstein spaces,
which are K\"ahler Ricci-flat.
This leads to the applications of
Hessian geometry \cite{Shima} on convex cones to AdS/CFT \cite{Mar-Spa-Yau}.
These techniques were generalized and applied to the Kepler problem in \cite{Zhou-Kepler}.
First, in \S 6 of that work,
we use the explicit construction of K\"ahler metrics
with $U(n)$-symmetry \cite{LeBrun, Duan-Zhou1, Duan-Zhou2} to
obtain applications of symplectic coordinates
and Hessian geometry.
Next,  these are applied to the $A_1$ case of Gibbons-Hawking metrics \cite{Gib-Haw}
(i.e., the Eguchi-Hanson metrics \cite{Egu-Han}), as the Kepler metric on the Kepler manifold $K_2$,
in \cite[\S 7.2]{Zhou-Kepler}.
An alternative treatment, based on Calabi ansatz \cite{Calabi},
is presented in \cite[\S 8]{Zhou-Kepler}.
Generalizations to the Kepler metric and the K\"ahler Ricci-flat metrics on the
resolved conifold are made in \S 9 and \S 10 of \cite{Zhou-Kepler},
based on Calabi ansatz on $\cO_{\bP^1}(-1) \oplus \cO_{\bP^1}(-1)$
and $\cO_{\bP^1\times \bP^1}(-1,-1)$ respectively.
The basic observations in \cite{Zhou-Kepler} is that all these metrics obtained
by such explicit constructions are toric
and the images of their moment maps are polyhedral convex bodies,
and one can recover the complex structures and K\"ahler structures
by Hessian structures \cite{Shima} on the convex bodies.
In this paper we will show the same holds for toric Gibbons-Hawking metrics.

As a direct consequence of the applications of moment maps and Hessian geometry,
we will study some phase change phenomena of the Gibbons-Hawking metrics.
We will describe a procedure in which ``black hole" appears naturally such that
outside the ``blackhole" we still have Riemannian metric
but inside it we get {\em imaginary} Riemannian metric.
From a string theoretical point of view this is very natural.
See the discussions in \S \ref{sec:Conclusions}.

The notion of a phase change for K\"ahler metrics were first introduced
in \cite{Duan-Zhou1, Duan-Zhou2}.
In \cite[\S 10.6]{Zhou-Kepler},
the author presented a new method to describe the flop of K\"ahler Ricci-flat
metrics on the resolved conifold \cite{Can-Oss}
by embedding them in a two-parameter family of K\"ahler Ricci-flat metrics
on the canonical line bundle of $\bP^1 \times \bP^1$.
It then becomes natural to reexamine the phase change phenomena introduced
in \cite{Duan-Zhou1, Duan-Zhou2} from the point of view of \cite{Zhou-Kepler}.
This has been carried out recently for local $\bP^1$, local $\bP^2$ and local $\bP^1 \times \bP^1$
cases in \cite{Wang-Zhou}.
The underlying spaces for these case are toric Calabi-Yau 3-folds $\cO_{\bP^1}(-1) \oplus \cO_{\bP^1}(-1)$,
$\cO_{\bP^2}(-3)$ and $\cO_{\bP^1 \times \bP^1}(-2, -2)$ respectively.
In this paper,
we will treat the case of toric Calabi-Yau 2-folds,
i.e.,
the crepant resolutions of $\bC^2/\bZ_n$.
They are also called the Gibbons-Hawking spaces,
or the ALE spaces of type A.
We conjecture that the phase change phenomena described in this paper can also
occur on ALE spaces of type D and type E.

The rest of the paper is arranged as follows.
In \S \ref{sec:GH},
after recalling the Gibbons-Hawking construction we present some
explicit choices of connection $1$-forms involved in this construction.
The applications of moment maps and Hessian geometry to toric
Gibbons-Hawking metrics are presented in \S \ref{sec:Hessian}.
Hessian geometry is used in \S \ref{sec:Complex} to explicitly
construct local complex coordinates on Gibbons-Hawking spaces
and to identify them with the crepant resolutions of $\bC^2/\bZ_n$.
We describe the phase changes for Gibbons-Hawking metrics in \S \ref{sec:Phase}.
We end the paper by some concluding remarks in \S \ref{sec:Conclusions}.

\section{Gibbons-Hawking Construction}
\label{sec:GH}

In this Section we first recall the Gibbons-Hawking construction \cite{Gib-Haw},
then present some explicit choices involved in the construction,
which will be crucial for computations in later Sections.

\subsection{Gibbons-Hawking construction}

Given $n$ distinct points $\vec{p}_1, \dots, \vec{p}_n$ in $\bR^3$,
consider a function $V$ defined by
\be
V(\vec{r}) = \half \sum_{j=1}^n \frac{1}{|\vec{r}-\vec{p}_j|}.
\ee
Clearly, $V$ is a solution of the Laplace equation
\be
\Delta V = 0
\ee
on $\bR^3-\{p_1, \dots, p_n\}$.
Denote by $\ast$ the Hodge star-operator on $\bR^3$.
Since $\Delta V = - * d*d V $,
we have
\be
d * dV = 0.
\ee
Let $U\subset \bR^3-\{p_1, \dots, p_n\}$ be a domain on which
we have
\be
*dV = - d\alpha.
\ee
On the principal bundle $U \times S^1 \to U$ with connection $1$-form $d\varphi+\alpha$,
where $\varphi$ is the natural coordinate on $S^1$, i.e., $e^{i\varphi} \in S^1$,
the horizontal lifts of
vector fields $\frac{\pd}{\pd x}$, $\frac{\pd}{\pd y}$ and $\frac{\pd}{\pd z}$
are
\begin{align}
\widehat{\frac{\pd}{\pd x}} & = \frac{\pd}{\pd x}
- \alpha(\frac{\pd}{\pd x}) \cdot \frac{\pd}{\pd \varphi}, \\
\widehat{\frac{\pd}{\pd x}} & = \frac{\pd}{\pd y}
- \alpha(\frac{\pd}{\pd y}) \cdot  \frac{\pd}{\pd \varphi}, \\
\widehat{\frac{\pd}{\pd z}} & = \frac{\pd}{\pd z}
- \alpha(\frac{\pd}{\pd x}) \cdot \frac{\pd}{\pd \varphi}.
\end{align}
Then
$$\biggl\{ V^{1/2} \frac{\pd}{\pd \varphi},V^{-1/2} \widehat{\frac{\pd}{\pd x}},
V^{-1/2}\widehat{\frac{\pd}{\pd y}}, V^{-1/2}\widehat{\frac{\pd}{\pd z}} \biggr\}$$
with the dual basis
\be
\{V^{-1/2} (d\varphi+\alpha), V^{1/2}dx, V^{1/2}dy, V^{1/2}dz\},
\ee
is a local orthonormal frame for the metric
\be
g = \frac{1}{V}(d\varphi + \alpha)^2 + V \cdot (dx^2+dy^2+dz^2).
\ee
This metric is the {\em Gibbons-Hawking metric} in local coordinates $\{\varphi, x, y, z\}$.

\subsection{Complex structures on Gibbons-Hawking spaces}

Consider the almost complex structure given by:
\begin{align*}
J\biggl(V^{1/2} \frac{\pd}{\pd \varphi}\biggr)
&= V^{-1/2} \widehat{\frac{\pd}{\pd z}}, &
J\biggl(V^{-1/2} \widehat{\frac{\pd}{\pd z}}\biggr)
& = - V^{1/2} \frac{\pd}{\pd \varphi}, \\
J\biggl(V^{-1/2} \widehat{\frac{\pd}{\pd x}}\biggr)
& = V^{-1/2} \widehat{\frac{\pd}{\pd y}}, &
J \biggl(V^{-1/2} \widehat{\frac{\pd}{\pd y}}\biggr) &
= - V^{-1/2} \widehat{\frac{\pd}{\pd x}}.
\end{align*}

The induced almost complex structure on the cotangent bundle in terms of
the orthonormal frame is given by:
\begin{align}
J^*(V^{-1/2}(d\varphi + \alpha)) & = - V^{1/2}dz, &
J^*(V^{1/2} d z) & = V^{-1/2}(d\varphi + \alpha), \\
J^*(V^{1/2}dx) & = - V^{1/2} dy, & J^*(V^{1/2}dy) & = V^{1/2} dx.
\end{align}
Therefore,
the space of type $(1, 0)$-forms are generated by:
\begin{align}
dx + \sqrt{-1} dy, && (d\varphi + \alpha) + \sqrt{-1}V dz.
\end{align}
One can compute their exterior differentials and
see that they do not contain $(0,2)$-components as follows:
\ben
&& d (dx + \sqrt{-1} dy) = 0, \\
&& d ((d\varphi + \alpha) + \sqrt{-1}{V} dz)
= d \alpha + \sqrt{-1} dV \wedge dz \\
& = & - *dV + \sqrt{-1} dV \wedge dz\\
& = & - \frac{\pd V}{\pd x} dy \wedge dz - \frac{\pd V}{\pd y} dz \wedge dx
- \frac{\pd V}{\pd z} dx \wedge dy \\
& + & \sqrt{-1} \frac{\pd V}{\pd x} dx \wedge dz
+ \sqrt{-1} \frac{\pd V}{\pd y} dy \wedge dz \\
& = & \sqrt{-1} \frac{\pd V}{\pd x} (dx + \sqrt{-1} dy) \wedge dz
+ \frac{\pd V}{\pd y}(dx + \sqrt{-1}dy ) \wedge dz \\
& - & \frac{\sqrt{-1}}{2} \frac{\pd V}{\pd z}(dx + \sqrt{-1}dy )
\wedge (dx - \sqrt{-1}dy ) \in \Omega^{2,0} \oplus \Omega^{1,1},
\een
therefore,
by Newlander-Nirenberg theorem,
the almost complex structure is integrable.
We will address the problem
of finding explicit local complex coordinates in \S \ref{sec:Complex}.

This complex structure is compatible with the Riemannian metric $g$, with
the symplectic form given by:
\be \label{eqn:Symplectic}
\omega = (d\varphi+\alpha) \wedge d z + V d x \wedge dy.
\ee
Since one has
\ben
d\omega & = & d\alpha \wedge dz + dV \wedge dx \wedge dy \\
& = & -*dV \wedge dz + \frac{\pd V}{\pd z} dz \wedge dx \wedge dy \\
& = & - \frac{\pd V}{\pd z} dx \wedge dy \wedge dz
+ \frac{\pd V}{\pd z} dz \wedge dx \wedge dy = 0,
\een
therefore,
$g, J$ gives a K\"ahler structure on the Gibbons-Hawking space.
Indeed, by changing $(x, y, z)$ to $(y,z,x)$ and $(z, x, y)$,
one gets two more K\"ahler structures, making the Gibbons-Hawking
metrics hyperk\"ahler.

\subsection{Explicit expressions}

Next we consider the explicit expressions.
Let $\vec{r} = (x, y, z)$ and $\vec{p}_j = (a_j, b_j, c_j)$,
then
\ben
&& dV = - \half \sum_{j=1}^n \frac{(x-a_j) dx + (y-b_j) dy + (z-c_j)dz}{|\vec{r}-\vec{p}_j|^3}, \\
&& * dV = - \half \sum_{j=1}^n \frac{(x-a_j) dy \wedge dz + (y-b_j) dz \wedge dx
+ (z-c_j)dx \wedge dy}{|\vec{r}-\vec{p}_j|^3}.
\een
Let us begin with the case of $n=1$ and let $c_1 = 0$,
i.e., $\vec{p}_1 = (0,0,0)$.
Then
\be
*dV =  - \half \frac{x dy \wedge dz + y dz \wedge dx + z dx \wedge dy}{\sqrt{(x^2+y^2+z^2)^3}}.
\ee
This is not an exact form on $\bR^3-\{(0,0,0)\}$,
but on $U:=\bR^3-\{(0,0,z)\;|\;z \geq 0\}$ we will use the stereographic
projection to get the following change of coordinates:
\begin{align}
x & = \frac{2ru}{u^2+v^2+1}, &
y & = \frac{2rv}{u^2+v^2+1}, &
z & = \frac{r(u^2+v^2-1)}{u^2+v^2+1}, \\
u & = \frac{x}{r-z}, &
v & = \frac{y}{r-z}, &
r & = \sqrt{x^2+y^2+z^2}.
\end{align}
Then we have
\be
*dV = \frac{2du \wedge dv}{(1+u^2+v^2)^2} = - d\alpha,
\ee
where
\be
\alpha = -\frac{u dv - v du}{1+u^2+v^2}
= - \half \frac{xdy-ydx}{r(r-z)}.
\ee
On $\tilde{U} = \bR^3-\{(0,0,z)\;|\;z \leq 0\}$ we will use spherical coordinates
\begin{align}
x & = \frac{2r\tilde{u}}{\tilde{u}^2+\tilde{v}^2+1}, &
y & = \frac{2r\tilde{v}}{\tilde{u}^2+\tilde{v}^2+1}, &
z & = \frac{r(1-\tilde{u}^2-\tilde{v}^2)}{\tilde{u}^2+\tilde{v}^2+1}, \\
\tilde{u} & = \frac{x}{r+z}, &
\tilde{v} & = \frac{y}{r+z}.
\end{align}
Then we have
\be
*dV = - \frac{2d\tilde{u} \wedge d\tilde{v}}{(1+\tilde{u}^2+\tilde{v}^2)^2} = -d\tilde{\alpha},
\ee
where
\be
\tilde{\alpha} = \frac{\tilde{u} d\tilde{v} - \tilde{v} d\tilde{u}}{1+\tilde{u}^2+\tilde{v}^2}
= \half \frac{xdy-ydx}{r(r+z)}..
\ee
On $U \cap \tilde{U}$ we have the following coordinate change formula:
\begin{align}
\tilde{u} & = \frac{u}{u^2+v^2}, & \tilde{v} & = \frac{v}{u^2+v^2}, \\
u & = \frac{\tilde{u}}{\tilde{u}^2+\tilde{v}^2}, &
v & = \frac{\tilde{v}}{\tilde{u}^2+\tilde{v}^2}.
\end{align}
One can then check that
\be
\tilde{\alpha} = \frac{\tilde{u} d\tilde{v} - \tilde{v} d\tilde{u}}{1+\tilde{u}^2+\tilde{v}^2}
= - \frac{vdu - udv}{(u^2+v^2) (1+u^2+v^2)}.
\ee
It is easy to check that
\be \label{eqn:tildeAlpha-alpha}
\tilde{\alpha} - \alpha = - \frac{vdu - u d v}{u^2+v^2} = \frac{xdy - y dx}{x^2+y^2}
= e^{-i \theta} \cdot d e^{i \theta},
\ee
where $\theta$ is the argument in the $(x,y)$-plane:
\be
\theta = \arctan \frac{y}{x}.
\ee

To get toric Gibbons-Hawking metrics,
we will let the points $p_1, \dots, p_n$ lie in a line, say,
$p_j = (0, 0, c_j)$, $j=1, \dots, n$,
$c_1 < c_2 < \cdots < c_n$,
we take $U:=\bR^3-\{(0,0,z)\;|\;z \geq c_1\}$ we take:
\begin{align}
x & = \frac{2r_ju_j}{u_j^2+v_j^2+1}, &
y & = \frac{2r_jv_j}{u_j^2+v_j^2+1}, &
z & = \frac{r_j(u_j^2+v_j^2-1)}{u_j^2+v_j^2+1}+c_j, \\
u_j & = \frac{x}{r_j-z+c_j}, &
v_j & = \frac{y}{r_j-z+c_j}, &
r_j & = \sqrt{x^2+y^2+(z-c_j)^2}.
\end{align}
Then we have
\be
*dV = \sum_{j=1}^n \frac{2du_j \wedge dv_j}{(1+u_j^2+v_j^2)^2} = - \sum_{j=1}^n d\alpha_j,
\ee
where
\be
\alpha_j = -\frac{u_j dv_j - v_j du_j}{1+u_j^2+v_j^2}
= - \half \frac{xdy-ydx}{r_j(r_j-z+c_j)}.
\ee
In particular,
\be \label{eqn:Symplectic2}
\omega = \biggl(d\varphi- \half \frac{xdy-ydx}{r_j(r_j-z+c_j)}\biggr) \wedge d z
+ \half \sum_{j=1}^n \frac{1}{r_j} d x \wedge dy.
\ee

On $\tilde{U} = \bR^3-\{(0,0,z)\;|\;z \leq c_n\}$ we will use spherical coordinates
\begin{align}
x & = \frac{2r_j\tilde{u}_j}{\tilde{u}_j^2+\tilde{v}_j^2+1}, &
y & = \frac{2r_j\tilde{v}_j}{\tilde{u}_j^2+\tilde{v}_j^2+1}, &
z & = \frac{r_j(1-\tilde{u}_j^2-\tilde{v}_j^2)}{\tilde{u}_j^2+\tilde{v}_j^2+1}, \\
\tilde{u}_j & = \frac{x}{r_j+z-c_j}, &
\tilde{v}_j & = \frac{y}{r_j+z-c_j}.
\end{align}
Then we have
\be
*dV = - \sum_{j=1}^n \frac{2d\tilde{u}_j \wedge d\tilde{v}_j}{(1+\tilde{u}_j^2+\tilde{v}_j^2)^2}
= - \sum_{j=1}^n d\tilde{\alpha}_j,
\ee
where
\be
\tilde{\alpha}_j = \frac{\tilde{u}_j d\tilde{v}_j
- \tilde{v}_j d\tilde{u}_j}{1+\tilde{u}_j^2+\tilde{v}_j^2}
= \half \frac{xdy-ydx}{r_j(r_j+z-c_j)}.
\ee
On $U \cap \tilde{U}$ we have the following coordinate change formula:
\begin{align}
\tilde{u} & = \frac{u}{u^2+v^2}, & \tilde{v} & = \frac{v}{u^2+v^2}, \\
u & = \frac{\tilde{u}}{\tilde{u}^2+\tilde{v}^2}, &
v & = \frac{\tilde{v}}{\tilde{u}^2+\tilde{v}^2}.
\end{align}
One can then check that
\be \label{eqn:tildeAlpha-Alpha-n}
\tilde{\alpha} - \alpha = - n \frac{vdu - u d v}{u^2+v^2}
= n d\theta = e^{-i n \theta} \cdot d e^{i n\theta} .
\ee

\section{Hessian Geometry of Toric Gibbons-Hawking Spaces}

\label{sec:Hessian}

In this Section we study the toric Gibbons-Hawking metrics
from the point of view of moment maps and Hessian geometry.
For this purpose,
we need to introduce a $2$-torus action and study its moment map.
This leads to Hessian structures  \cite{Shima}
on the convex bodies which arise as the image of the moment map.

\subsection{Torus action and moment map on a toric Gibbons-Hawking space}

We now look at the $2$-torus  action given in local coordinates by:
\be \label{eqn:action}
(e^{i\theta_1}, e^{i\theta_2})  \cdot (\varphi, x, y, z)
= (\varphi+ \theta_1, x\cos\theta_2 - y \sin \theta_2, x\sin \theta_2+y\cos \theta_2, z),
\ee
the associated vector field is
$X_1 = \frac{\pd}{\pd \varphi}$ and $X_2 = -y \frac{\pd}{\pd x} + x \frac{\pd}{\pd y}$.
From the formula \eqref{eqn:Symplectic} for the symplectic form,
it is easy to see that the moment map with respect to the torus action above is given by:
\be
\mu = (\mu_1, \mu_2) = (-z, \half\sum_{j=1}^n (r_j+z-c_j)),
\ee
where  $r_j$ is defined by
\be
r_j = \sqrt{x^2+y^2+(z-c_j)^2}.
\ee
Indeed, we have
\bea
d \mu_1 & = &- d z = -i_{\frac{\pd}{\pd \varphi}}\omega =- i_{X_1} \omega,
\label{eqn:dmu1} \\
d \mu_2 
& = & \half \sum_{i=1}^n (\frac{xdx+ydy+(z-c_j)dz}{r_j} + dz)
= -i_{X_2} \omega. \label{eqn:dmu2}
\eea

\subsection{Image of the moment map}
Note:
\ben
\mu_2& = & \half \sum_{j=1}^n (\sqrt{x^2+y^2+(z-c_j)^2} + z-c_j) \\
& \geq & \half \sum_{j=1}^n (|z-c_j|+z-c_j),
\een
and since
\be
|z-c_j|+z-c_j = \begin{cases}
0, & \text{if $z\leq c_j$}, \\
2(z-c_j), & \text{if $z \geq c_j$},
\end{cases}
\ee
it is easy to see that the image of the moment map is the convex region given by
the following inequalities:
\ben
&& l_0: = \mu_2 \geq 0, \\
&& l_1: = \mu_2 + (\mu_1 + c_1) \geq 0, \\
&& l_2: = \mu_2 + (\mu_1 + c_1) + (\mu_1 + c_2) \geq 0, \\
&&  \cdots\cdots \cdots \cdots \cdots \cdots \\
&& l_n: = \mu_2 + \sum_{j=1}^n (\mu_1 + c_j) \geq 0.
\een

\subsection{Symplectic coordinates and Hessian geometry for toric
Gibbons-Hawking spaces}

The following is our first main result:

\begin{prop}
If one takes the following local coordinates:
\begin{align}
\theta_1 & = \varphi, & \mu_1 & = -z, \\
\theta_2 & = \arctan \frac{y}{x}, &
\mu_2 & = \half\sum_{j=1}^n (r_j+z-c_j),
\end{align}
then the symplectic form takes the following form:
\be \label{eqn:Symplectic3}
\omega 
= d\mu_1 \wedge d\theta_1  + d\mu_2 \wedge d\theta_2.
\ee
And in these symplectic coordinates,
the Gibbons-Hawking metric takes the following form:
\be \label{eqn:g-Hessian}
\begin{split}
g &  = \frac{1}{V} d\theta_1^2 + \frac{1}{V}
\sum_{j=1}^n\frac{\rho^2}{r_j(r_j-(z-c_j))} d\theta_1 d\theta_2 \\
& + \biggl[V \rho^2 +
\frac{1}{4V} \biggl(\sum_{j=1}^n\frac{\rho^2}{r_j(r_j-(z-c_j))}\biggr)^2 \biggr] d\theta_2^2 \\
& +  \biggl[V  +
\frac{\rho^2}{4V} \biggl(\sum_{j=1}^n\frac{1}{r_j(r_j-(z-c_j))}\biggr)^2 \biggr] d\mu_1^2 \\
& -   \frac{1}{V}
\sum_{j=1}^n\frac{1}{r_j(r_j-(z-c_j))} d\mu_1 d\mu_2 + \frac{1}{V\rho^2} d\mu_2^2,
\end{split}
\ee
where $\rho^2 = x^2 + y^2$.
The complex potential and the K\"ahler potential are given by the following formulas respectively:
\be \label{eqn:Complex}
\begin{split}
\psi & = \half \sum_{j=1}^n \biggl( (r_j + (z-c_j)) \log (r_j + (z-c_j)) \\
& + (r_j-(z-c_j)) \log(r_j-(z-c_j)) \biggr) +C_1\mu_1 + C_2\mu_2.
\end{split}
\ee
And the K\"ahler potential is given by:
\be \label{eqn:Kahler}
\psi^\vee= -\sum_{j=1}^n c_j \log(r_j-(z-c_j))+C_1\mu_1 + C_2\mu_2.
\ee
for some constants $C_1$ and $C_2$.
\end{prop}

\begin{proof}
By \eqref{eqn:Symplectic2}, \eqref{eqn:dmu1} and \eqref{eqn:dmu2},
it is straightforward to get \eqref{eqn:Symplectic3}.
By \eqref{eqn:dmu2} we have
\ben
d\rho & = & \frac{xdx+ydy}{\rho}
= \frac{1}{V} \biggl( d\mu + \half \sum_{j=1}^n \frac{r_j+z-c_j}{r_j} d\mu_1\biggr),
\een
and so we have
\ben
g & = & \frac{1}{V}\biggl(d\varphi -\half \sum_{j=1}^n
\frac{xdy-ydx}{r_j(r_j-(z-c_j))}\biggr)^2 + V \cdot (d\rho^2+\rho^2d\theta_2^2+dz^2) \\
& = & \frac{1}{V}\biggl(d\theta_1 - \half \sum_{j=1}^n
\frac{\rho^2d\theta_2}{r_j(r_j-(z-c_j))}\biggr)^2 + V \cdot (d\rho^2+\rho^2d\theta_2^2+d\mu_1^2) \\
& = & \frac{1}{V}\biggl(d\theta_1 -\half \sum_{j=1}^n
\frac{\rho^2d\theta_2}{r_j(r_j-(z-c_j))}\biggr)^2 + V \rho^2d\theta_2^2 \\
& + & \frac{1}{V\rho^2} \biggl(d\mu_2+\half \sum_{j=1}^n \frac{r_j+z-c_j}{r_j} d\mu_1\biggr)^2
+ V d\mu_1^2 \\
& = & \frac{1}{V} d\theta_1^2 - \frac{1}{V}
\sum_{j=1}^n\frac{\rho^2}{r_j(r_j-(z-c_j))} d\theta_1 d\theta_2 \\
& + & \biggl[V \rho^2 +
\frac{1}{4V} \biggl(\sum_{j=1}^n\frac{\rho^2}{r_j(r_j-(z-c_j))}\biggr)^2 \biggr] d\theta_2^2 \\
& + & \biggl[V  +
\frac{\rho^2}{4V} \biggl(\sum_{j=1}^n\frac{1}{r_j(r_j-(z-c_j))}\biggr)^2 \biggr] d\mu_1^2 \\
& + &  \frac{1}{V} \sum_{j=1}^n\frac{1}{r_j(r_j-(z-c_j))} d\mu_1 d\mu_2
+ \frac{1}{V\rho^2} d\mu_2^2.
\een
The formulas for the complex potential and the K\"ahler potential will be proved
in the next two subsections.
\end{proof}

Similar to \cite{Zhou-Kepler},
write:
\be
g = \sum_{i,j=1}^2 (\frac{1}{2} G_{ij} d\mu_i d\mu_j +2 G^{ij}d\theta_id\theta_j),
\ee
where the coefficient matrices $(G_{ij})_{i,j=1, 2}$ and $(G^{ij})_{i,j=1,2}$ are given by:
\ben
(G_{ij})_{i,j=1, 2}
= \begin{pmatrix}
2V  +
\frac{\rho^2}{2V} \biggl(\sum_{j=1}^n\frac{\rho^2}{r_j(r_j-(z-c_j))}\biggr)^2 &
\frac{1}{V} \sum_{j=1}^n\frac{1}{r_j(r_j-(z-c_j))} \\
\frac{1}{V} \sum_{j=1}^n\frac{1}{r_j(r_j-(z-c_j))} &
\frac{2}{V\rho^2}
\end{pmatrix}
\een

\ben
( G^{ij})_{i,j=1,2}=
\begin{pmatrix}
\frac{1}{2V} & -\frac{1}{4V} \sum_{j=1}^n\frac{\rho^2}{r_j(r_j-(z-c_j))}  \\
-\frac{1}{4V} \sum_{j=1}^n\frac{\rho^2}{r_j(r_j-(z-c_j))}
& \frac{V \rho^2}{2} +
\frac{1}{8V} \biggl(\sum_{j=1}^n\frac{\rho^2}{r_j(r_j-(z-c_j))}\biggr)^2
\end{pmatrix}
\een
It is easy to see that these matrices are inverse to each other.

\subsection{Complex potential functions}

Next we will show that
\be
G_{ij} = \frac{\pd^2\psi}{\pd \mu_1\pd \mu_2}
\ee
for some function $\psi$.
To find $\psi$, we rewrite \eqref{eqn:dmu2} as follows:
\be \label{eqn:dmu2-2}
d\mu_2 = V \rho d\rho + \half \sum_{j=1}^n \frac{r_j+z-c_j}{r_j} dz.
\ee
From this we get:
\be \label{eqn:drho}
d\rho = \frac{1}{V\rho} \biggl(d\mu_2+\half \sum_{j=1}^n \frac{r_j+z-c_j}{r_j} d\mu_1\biggr).
\ee
It follows that
\be \label{eqn:Complex1}
\begin{split}
& \biggl[V +
\frac{\rho^2}{4V} \biggl(\sum_{j=1}^n\frac{1}{r_j(r_j-(z-c_j))}\biggr)^2 \biggr] d\mu_1
+ \half \frac{1}{V} \sum_{j=1}^n\frac{1}{r_j(r_j-(z-c_j))} d\mu_2 \\
= & -V dz + \frac{1}{2} \sum_{j=1}^n\frac{\rho d\rho}{r_j(r_j-(z-c_j))}
= \half d \sum_{j=1}^n  \log (r_j - (z-c_j)),
\end{split}
\ee
and
\be \label{eqn:Complex2}
\half \frac{1}{V}
\sum_{j=1}^n\frac{1}{r_j(r_j-(z-c_j))} d\mu_1  + \frac{1}{V\rho^2} d\mu_2
= \frac{d\rho}{\rho} =  d \log \rho.
\ee
Therefore,
one has
\begin{align} \label{eqn:partial-psi}
\frac{\pd \psi}{\pd \mu_1}
& = \sum_{j=1}^n  \log (r_j - (z-c_j)) + C_1, &
\frac{\pd \psi}{\pd \mu_2}
& = 2 \log \rho + C_2.
\end{align}
Furthermore,
\ben
&& \half \sum_{j=1}^n  \log (r_j - (z-c_j)) d\mu_1 + \log \rho d\mu_2 \\
& = & -\half \sum_{j=1}^n  \log (r_j - (z-c_j)) dz
+ V \rho \log \rho d\rho +  \half \sum_{j=1}^n \frac{r_j+z-c_j}{r_j} \log \rho dz \\
& = & \half \sum_{j=1}^n \biggl(\frac{r_j+z-c_j}{r_j} \log \rho
- \log (r_j - (z-c_j))  \biggr) dz + V \rho \log \rho d\rho \\
& = & \half d   \sum_{j=1}^n \biggl( (r_j + z-c_j) \log \rho
- (z-c_j) \log(r_j-(z-c_j)) - r_j \biggr).
\een
So we can take
\be
\psi =  \sum_{j=1}^n \biggl( (r_j + z-c_j) \log \rho
- (z-c_j) \log(r_j-(z-c_j)) - r_j \biggr) + C_1\mu_1 + C_2\mu_2,
\ee
so that the equalities in \eqref{eqn:partial-psi} hold.
We can write $\psi$ in another form. Since
\be
\sum_{j=1}^n r_j = 2\mu_2- \sum_{j=1}^n (z-c_j)
\ee
is linear in $\mu_1, \mu_2$,
we can take:
\be
\psi = \sum_{j=1}^n \biggl( (r_j + z-c_j) \log \rho
- (z-c_j) \log(r_j-(z-c_j)) \biggr).
\ee
Using the equalities
\be
\rho^2 = r_j^2 - (z-c_j)^2,
\ee
we have
\be \label{eqn:Complex}
\begin{split}
\psi & = \half \sum_{j=1}^n \biggl( (r_j + (z-c_j)) \log (r_j + (z-c_j)) \\
& + (r_j-(z-c_j)) \log(r_j-(z-c_j)) \biggr)
+ C_1 \mu_1 + C_2 \mu_2.
\end{split}
\ee

\subsubsection{The $n=1$ case} One can take $c_1=0$.
From the equation
\be
\sqrt{\rho^2+z^2} + z = 2 \mu_2
\ee
we can solve for $\rho^2$:
\be
\rho^2 = 4\mu_2 (\mu_2-z)
\ee
and
\be
r_1 = 2\mu_2-z.
\ee
The complex potential is
\be \label{eqn:Complex-n=1}
\psi = \mu_2\log \mu_2 + (\mu_2+\mu_1) \log(\mu_2+\mu_1) +C_1\mu_1 + C_2\mu_2
\ee
for some constants $C_1, C_2$,
and the K\"ahler potential is
\be
\psi^\vee = \mu_2 + (\mu_2+\mu_1) + C_1\mu_1+C_2\mu_2.
\ee

Make the following change of coordinates:
\begin{align}
y_1 & = \mu_2, &
y_2 & = \mu_2 +\mu_1.
\end{align}
Then the image of the moment map is changed from the convex cone defined by:
\begin{align}
\mu_2 & \geq 0, & \mu_2 +\mu_1 & \geq 0
\end{align}
to the convex cone defined by:
\begin{align}
y_1 & \geq 0, & y_2 & \geq 0
\end{align}
and the complex potential function becomes
\be
\psi = y_1\log y_1 + y_2 \log y_2 + C_1'y_1 + C_2'y_2
\ee
for some constants $C_1', C_2'$.

\subsubsection{The $n=2$ case}
From the equation
\be
\sqrt{\rho^2+(z-c_1)^2} + (z-c_1) + \sqrt{\rho^2 + (z-c_2)^2} + (z-c_2)
= 2 \mu_2
\ee
we can solve for $\rho^2$:
\be \label{rho2-in-mu}
\rho^2 = \frac{4\mu_2 (\mu_2-(z-c_1))(\mu_2-(z-c_2))(\mu_2-(z-c_1)-(z-c_2))}
{(2\mu_2 -(z-c_1)-(z-c_2))^2}.
\ee
From this one finds that
\bea
&& r_1
= \frac{2\mu_2(\mu_2-(z-c_1)-(z-c_2)) + (z-c_1)(z-c_1 + z-c_2)}{2\mu_2-(z-c_1)-(z-c_2)},
\label{eqn:r1-in-mu} \\
&& r_2 = \frac{2\mu_2(\mu_2-(z-c_1)-(z-c_2)) + (z-c_1)(z-c_1 + z-c_2)}{2\mu_2-(z-c_1)-(z-c_2)},
\label{eqn:r2-in-mu}
\eea
and by these the complex potential function becomes:
\be \label{eqn:Complex-n=2}
\begin{split}
\psi & = \mu_2 \log (\mu_2) + \sum_{j=1}^2 (\mu_2 -(z-c_j)) \log (\mu_2-(z-c_j))  \\
& + (\mu_2-z+c_1-z+c_2) \log (\mu_2 -z+c_1 -z+c_2) \\
& - (2\mu_2-z+c_1-z+c_2) \log (2\mu_2 -z+c_1 -z+c_2)\\
& + C_1\mu_1+C_2\mu_2
\end{split}
\ee
for some constants $C_1, C_2$.
We make the following change of coordinates:
\begin{align}
y_1 & = \mu_2, &
y_2 & = \mu_2 - 2z + c_1 + c_2, & b & = c_2 - c_1.
\end{align}
Then the image of the moment map is changed from the convex cone defined by:
\begin{align}
\mu_2 & \geq 0, & \mu_2 - (z-c_1) & \geq 0, & \mu_2 - (z-c_1) - (z-c_2) & \geq 0
\end{align}
to the convex cone defined by:
\begin{align}
y_1 & \geq 0, & y_1+y_2 & \geq b, & y_2 & \geq 0,
\end{align}
and up to a linear function in $y_1, y_2$,
\be
\psi = y_1 \log y_1 + y_2 \log y_2 - y \log y + \half (y -b) \log (y-b) + \half (y+b) \log (y+b),
\ee
where $y=y_1+y_2$,
up to a term of the form $C_1'y_1 + C_2'y_2$ for some constants $C_1', C_2'$.
Later we will fix these constants to be zero, so
this is a special case of
the following formula for the complex potential
of K\"ahler Ricci-flat metrics on $\cO_{\bP^{n-1}}(-n)$ derived in \cite[\S 7.1]{Zhou-Kepler}:
\be
\psi = \sum_{i=1}^n y_i (\ln y_i - 1)
- y (\ln y -1) + \frac{1}{n} \sum_{j=0}^{n-1} (y-b\xi_n^j) (\log (y- b\xi_n^j) - 1)-C.
\ee
For $n=2$, this gives us the Eguchi-Hanson metric \cite{Egu-Han}.

\subsection{Legendre transform and K\"ahler potential}

As in \cite{Zhou-Kepler},
the K\"ahler potentials of the toric Gibbons-Hawking metrics are:
\ben
\psi^\vee
& = & \frac{\pd \psi}{\pd \mu_1} \mu_1 + \frac{\pd \psi}{\pd \mu_2} \mu_2 - \psi \\
& = &  - \sum_{j=1}^n  \log (r_j - (z-c_j)) \cdot z  + 2 \log \rho \cdot \mu_2
+C_1\mu_1 + C_2\mu_2
- \psi \\
& = & -\sum_{j=1}^n  z \cdot \log (r_j - (z-c_j))
+  \sum_{j=1}^n (r_j+z-c_j)\cdot \log \rho  +C_1\mu_1 + C_2\mu_2\\
& - & \sum_{j=1}^n \biggl( (r_j + z-c_j) \log \rho
- (z-c_j) \log(r_j-(z-c_j)) \biggr) \\
& = & -\sum_{j=1}^n c_j \log(r_j-(z-c_j)) +C_1\mu_1 + C_2\mu_2.
\een

\section{Local Complex coordinates on Toric Gibbons-Hawking Spaces via Hessian Geometry}

\label{sec:Complex}

In this Section we study local complex coordinates and their relationships which arise naturally
from the point of Hessian geometry.
To arrive at the general case,
we need to first study the $n=1$ and $n=2$ cases in detail.

\subsection{Hessian local complex coordinates}

The function $\psi$ is called the {\em complex potential}
because one can find local complex coordinates $z_1$ and $z_2$ so that
\be
\frac{dz_i}{z_i} = \half \sum_{j=1}^2 \frac{\pd^2\psi}{\pd \mu_i \pd \mu_j} d\mu_j
+ \sqrt{-1} d\theta_i =  \half \sum_{i,j=1}^2 G_{ij} d\mu_j + \sqrt{-1} d \theta_i
\ee
is of type $(1, 0)$.
By \eqref{eqn:Complex1} and \eqref{eqn:Complex2},
we have
\bea
\frac{dz_1}{z_1} & = & 
\half d \sum_{j=1}^n  \log (r_j - (z-c_j))
+ \sqrt{-1} d\theta_1, \label{eqn:dz1/z1} \\
\frac{dz_2}{z_2} & = & d \log \rho + \sqrt{-1} d\theta_2 = d \log (x + y \sqrt{-1}).
\label{eqn:dz2/z2}
\eea
Therefore, we take
\bea
&& z_1 = \prod_{j=1}^n (r_j-(z-c_j))^{1/2}\cdot e^{\sqrt{-1} \theta_1},
\label{eqn:z1} \\
&& z_2 = x+ \sqrt{-1} y. \label{eqn:z2}
\eea

\subsection{Toric Gibbons-Hawking metrics and K\"ahler form
in Hessian local complex coordinates}

We now express Riemannian metric, symplectic form and K\"ahler potential
in terms of these local complex coordinates.
First of all,
by \eqref{eqn:dz1/z1} and \eqref{eqn:dz2/z2},
\bea
&& d\theta_1 = \frac{1}{2\sqrt{-1}}
\biggl( \frac{dz_1}{z_1} - \frac{d\bar{z}_1}{\bar{z}_1} \biggr), \\
&& d\theta_2 = \frac{1}{2\sqrt{-1}}
\biggl( \frac{dz_2}{z_2} - \frac{d\bar{z}_2}{\bar{z}_2} \biggr), \\
&& V dz - \frac{1}{2} \sum_{j=1}^n\frac{\rho d\rho}{r_j(r_j-(z-c_j))}
=  \Red{-} \frac{1}{2} \biggl( \frac{dz_1}{z_1} + \frac{d\bar{z}_1}{\bar{z}_1} \biggr), \\
&& d \log \rho = \frac{1}{2} \biggl( \frac{dz_2}{z_2} + \frac{d\bar{z}_2}{\bar{z}_2} \biggr).
\eea
From the last two equality we derive:
\be
d\mu_1 = -dz = \frac{1}{2V} \biggl( \frac{dz_1}{z_1} + \frac{d\bar{z}_1}{\bar{z}_1} \biggr)
-  \frac{1}{4V} \sum_{j=1}^n\frac{r_j+z-c_j}{r_j}
\biggl( \frac{dz_2}{z_2} + \frac{d\bar{z}_2}{\bar{z}_2} \biggr).
\ee
Combined with \eqref{eqn:dmu2}, we get:
\be
\begin{split}
d\mu_2 & = - \frac{1}{4V} \sum_{j=1}^n \frac{r_j+z-c_j}{r_j}
 \biggl( \frac{dz_1}{z_1} + \frac{d\bar{z}_1}{\bar{z}_1} \biggr) \\
& + \biggl[  \frac{V \rho^2}{2} + \frac{1}{8V}
\biggl(\sum_{j=1}^n \frac{r_j+z-c_j}{r_j} \biggr)^2 \biggr]
\biggl( \frac{dz_2}{z_2} + \frac{d\bar{z}_2}{\bar{z}_2} \biggr).
\end{split}
\ee

\begin{theorem}
The toric Gibbons-Hawking metrics and K\"ahler forms are given in complex coordinates $z_1, z_2$
as follows:
\be \label{eqn:g-in-complex}
\begin{split}
g = &   \frac{1}{V} \frac{dz_1}{z_1}\frac{d\bar{z}_1}{\bar{z}_1}
+ \frac{1}{2V} \sum_{j=1}^n\frac{r_j+z-c_j}{r_j}
\biggl( \frac{dz_1}{z_1} \frac{d\bar{z}_2}{\bar{z}_2}
+ \frac{dz_2}{z_2} \frac{d\bar{z}_1}{\bar{z}_1} \biggr) \\
+ & \biggl[ V \rho^2 + \frac{1}{4V}
\biggl( \sum_{j=1}^n \frac{r_j+z-c_j}{r_j}\biggr)^2 \biggr]\emph{}
\frac{dz_2}{z_2} \frac{d\bar{z}_2}{\bar{z}_2} ,
\end{split}
\ee
\be
\begin{split}
\omega & = \frac{1}{2\sqrt{-1}} \biggl(\frac{1}{V} \frac{dz_1}{z_1} \wedge
\frac{d\bar{z}_1}{\bar{z}_1} \\
& + \frac{1}{2V} \sum_{j=1}^n\frac{r_j+z-c_j}{r_j}
\biggl( \frac{dz_1}{z_1} \wedge \frac{d\bar{z}_2}{\bar{z}_2}
+ \frac{dz_2}{z_2}\wedge  \frac{d\bar{z}_1}{\bar{z}_1} \biggr) \\
& + \biggl[ V \rho^2 + \frac{1}{4V}
\biggl( \sum_{j=1}^n \frac{r_j+z-c_j}{r_j}\biggr)^2 \biggr]
\frac{dz_2}{z_2} \wedge \frac{d\bar{z}_2}{\bar{z}_2} \biggr).
\end{split}
\ee
\end{theorem}

\begin{proof}
These can verified by straightforward computations as follows:
\ben
g & = & \frac{1}{V}\biggl(d\theta_1 - \half \sum_{j=1}^n
\frac{\rho^2d\theta_2}{r_j(r_j-(z-c_j))}\biggr)^2 + V \cdot (d\rho^2+\rho^2d\theta_2^2+dz^2) \\
& = & \frac{1}{V}(d\theta_1)^2
- \frac{1}{V} \sum_{j=1}^n\frac{r_j+z-c_j}{r_j} d\theta_1d\theta_2
+ \frac{1}{4V}\biggl( \sum_{j=1}^n
\frac{r_j+z-c_j}{r_j}\biggr)^2 (d\theta_2)^2 \\
& + & Vdz_2 d\bar{z}_2
+ V \biggl[\frac{1}{2V} \biggl( \frac{dz_1}{z_1} + \frac{d\bar{z}_1}{\bar{z}_1} \biggr)
- \frac{1}{4V} \sum_{j=1}^n\frac{r_j+z-c_j}{r_j}
\biggl( \frac{dz_2}{z_2} + \frac{d\bar{z}_2}{\bar{z}_2} \biggr) \biggr]^2 \\
& = & \frac{1}{V} \frac{dz_1}{z_1}\frac{d\bar{z}_1}{\bar{z}_1}
- \frac{1}{2V} \sum_{j=1}^n\frac{r_j+z-c_j}{r_j}
\biggl( \frac{dz_1}{z_1} \frac{d\bar{z}_2}{\bar{z}_2}
+ \frac{dz_2}{z_2} \frac{d\bar{z}_1}{\bar{z}_1} \biggr) \\
& + & \biggl[ V \rho^2 + \frac{1}{4V}
\biggl( \sum_{j=1}^n \frac{r_j+z-c_j}{r_j}\biggr)^2 \biggr]
\frac{dz_2}{z_2} \frac{d\bar{z}_2}{\bar{z}_2},
\een

\ben
\omega
& = & d \mu_1  \wedge d\theta_1 +  d \mu_2  \wedge d\theta_2 \\
& = & - \biggl[ \frac{1}{2V} \biggl( \frac{dz_1}{z_1} + \frac{d\bar{z}_1}{\bar{z}_1} \biggr)
-  \frac{1}{4V} \sum_{j=1}^n\frac{r_j+z-c_j}{r_j}
 \biggl( \frac{dz_2}{z_2} + \frac{d\bar{z}_2}{\bar{z}_2} \biggr)\biggr]
\wedge  \frac{1}{2\sqrt{-1}}
\biggl( \frac{dz_1}{z_1} - \frac{d\bar{z}_1}{\bar{z}_1} \biggr) \\
& + &  \biggl[ - \frac{1}{4V} \sum_{j=1}^n \frac{r_j+z-c_j}{r_j}
 \biggl( \frac{dz_1}{z_1} + \frac{d\bar{z}_1}{\bar{z}_1} \biggr) \\
& + & \biggl[  \frac{V \rho^2}{2} + \frac{1}{8V}
\biggl(\sum_{j=1}^n \frac{r_j+z-c_j}{r_j} \biggr)^2 \biggr]
\biggl( \frac{dz_2}{z_2} + \frac{d\bar{z}_2}{\bar{z}_2} \biggr)\biggr]
\wedge \frac{1}{2\sqrt{-1}}
\biggl( \frac{dz_2}{z_2} - \frac{d\bar{z}_2}{\bar{z}_2} \biggr) \\
& = & \frac{1}{2\sqrt{-1}} \biggl[ \frac{1}{V}
\frac{dz_1}{z_1} \wedge \frac{d\bar{z}_1}{\bar{z}_1}
- \frac{1}{2V} \sum_{j=1}^n\frac{r_j+z-c_j}{r_j}
\biggl(\frac{dz_1}{z_1} \wedge \frac{d\bar{z}_2}{\bar{z}_2}
+ \frac{dz_2}{z_2} \wedge \frac{d\bar{z}_1}{\bar{z}_1}  \biggr) \\
& + & \biggl(  V \rho^2 + \frac{1}{4V}
\biggl(\sum_{j=1}^n \frac{r_j+z-c_j}{r_j} \biggr)^2 \biggr)
\frac{dz_2}{z_2} \wedge \frac{d\bar{z}_2}{\bar{z}_2} \biggr].
\een
\end{proof}

\subsection{The $n=1$ case}
From
\begin{align}
z_1 & = ((x^2+y^2+z^2)^{1/2}-z)^{1/2}\cdot e^{\sqrt{-1} \theta_1},
& z_2 & = x+ \sqrt{-1} y.
\end{align}
we get:
\begin{align}
\rho^2 & = |z_2|^2, & r & = \half(|z_1|^{-2}|z_2|^2+|z_1|^{2}), &
z & = \frac{1}{2}(|z_1|^{-2}|z_2|^2-|z_1|^{2}).
\end{align}
And so
\be
\mu_2 = \half(r+z) = \half |z_1|^{-2}|z_2|^2.
\ee
When $z_2=0$,
we have
\begin{align}
\mu_2 & = 0, & \mu_1 & = - z = \half |z_1|^2 > 0.
\end{align}

By \eqref{eqn:g-in-complex} we then have:
\ben
g & = & (1+ |z_1|^{-4}|z_2|^2) dz_1  d \bar{z}_1
- |z_1|^{-2}|z_2|^2 \biggl( \frac{dz_1}{z_1} \frac{d\bar{z}_2}{\bar{z}_2}
+ \frac{dz_2}{z_2} \frac{d\bar{z}_1}{\bar{z}_1} \biggr) \\
& + & \biggl(\frac{|z_2|^2}{|z_1|^{-2}|z_2|^2+|z_1|^{2}}
+  \frac{|z_1|^{-4}|z_2|^4}{|z_1|^{-2}|z_2|^2+|z_1|^2}
\biggr)\frac{dz_2}{z_2} \frac{d\bar{z}_2}{\bar{z}_2}
\een

\subsubsection{The $(\alpha, \beta)$-coordinates}

Make the following change of variables:
\be \label{eqn:Alpha-Beta}
z_1 = \beta,  \;\;\;\;\; z_2 = \alpha\beta.
\ee
The Gibbons-Hawking metric becomes:
\ben
g & = &  (1+ |\beta|^{-4}|\alpha\beta|^2) d\beta  d \bar{\beta} \\
& - &  |\beta|^{-2}|\alpha\beta|^2
\biggl( \frac{d\beta}{\beta}
\biggl(\frac{d\bar{\alpha}}{\bar{\alpha}}
+  \frac{d\bar{\beta}}{\bar{\beta}}  \biggr)
+ \biggl(\frac{d\alpha}{\alpha}
+ \frac{d\beta}{\beta}  \biggr)
\frac{d\bar{\beta}}{\bar{\beta}} \biggr) \\
& + &  |\alpha\beta|^2|\beta|^{-2}
\biggl(\frac{d\alpha}{\alpha} + \frac{d\beta}{\beta}  \biggr)
\biggl(\frac{d\bar{\alpha}}{\bar{\alpha}}
+  \frac{d\bar{\beta}}{\bar{\beta}}  \biggr) \\
& = & d\alpha d\bar{\alpha} + d\beta d\bar{\beta}.
\een
The K\"ahler form is
\be
\omega = \frac{\sqrt{-1}}{2} (d\alpha \wedge d\bar{\alpha}
+ d\beta \wedge d\bar{\beta}),
\ee
and the K\"ahler potential can be taken to be
\be
\psi^\vee = \half |\alpha|^2+\half |\beta|^2
\ee
One can check that
\begin{align}
x + y \sqrt{-1} & = \alpha \beta, & z & =\half (|\alpha|^2 - |\beta|^2), &
\mu_2 & = \half |\alpha|^2.
\end{align}
And so we get
$$\psi^\vee = \mu_2 + (\mu_2+\mu_1).$$
This fixes the constants $C_1=C_2 = 0$ in \eqref{eqn:Complex-n=1}.
Also note
\begin{align*}
z|_{\alpha = 0}  & = - \half |\beta|^2 \leq 0, &
\mu_2|_{\alpha = 0} & = 0,
\end{align*}
and so the $\beta$-axis in the $(\alpha, \beta)$-plane is mapped
by the moment map to
$$L_1= \{(\mu_1, \mu_2)\;|\; \mu_1 \geq 0, \mu_2 =0 \};$$
similarly,
\begin{align*}
z|_{\beta = 0}  & = \half |\alpha|^2 \geq 0, & \mu_2|_{\beta=0} = \half |\alpha|^2 \geq 0,
\end{align*}
i.e., the $\alpha$-axis in the $(\alpha, \beta)$-plane is mapped to
$$L_2 =\{(\mu_1, - \mu_1)\;|\; \mu_1 \leq 0\}.$$

\subsubsection{Change of Hessian local complex coordinates in the $n=1$ case}

In the above we have focused on the local coordinate patch over $U = \bR^3 -
\{ (0, 0, z) \;|\; z \geq 0 \}$,
now we switch the local coordinate patch over $\tilde{U} = \bR^3 - \{ (0,0,z) \;|\; z \leq 0\}$
where we have
\be
g = \frac{1}{V}(d\tilde{\varphi} + \tilde{\alpha})^2 + V \cdot (dx^2+dy^2+dz^2).
\ee
Recall in \eqref{eqn:tildeAlpha-alpha} we have seen that $\tilde{\alpha} = \alpha + \theta_2$,
there we can take
\be
\tilde{\varphi} = \varphi - \arctan\frac{y}{x} = \varphi -\emph{} \theta_2,
\ee
so that $d\tilde{\varphi} + \tilde{\alpha} = d \varphi + \alpha$.
Now the symplectic form can be written as
\ben
\omega & = & (d\tilde{\varphi}+\tilde{\alpha}) \wedge d z + V d x \wedge dy
= d\tilde{\mu}_1 \wedge d\tilde{\theta}_1 + d \tilde{\mu}_2 \wedge d \tilde{\theta}_2,
\een
where we can take
\begin{align}
\tilde{\theta}_1 & = \tilde{\varphi} = \theta_1-\theta_2,
& \tilde{\theta}_2 & = \theta_2 =  \arctan \frac{y}{x}, \\
\tilde{\mu}_1 & = \mu_1 = - z,   &
\tilde{\mu}_2 & = \mu_2 - z =  \half (r - z).
\end{align}
Indeed
\ben
\omega & = & d\varphi \wedge dz + d \mu_2 \wedge d\theta_2  \\
& = & (d\tilde{\varphi} + d\tilde{\theta}_2) \wedge dz + d \mu_2 \wedge d\tilde{\theta}_2 \\
& = & -dz \wedge d \tilde{\varphi} +  d (\mu_2 - z) \wedge  d \tilde{\theta}_2.
\een
From the complex potential
 one can find local complex coordinates $\tilde{z}_1$ and $\tilde{z}_2$ so that
\be
\half \sum_{j=1}^2 \frac{\pd^2\psi}{\pd \tilde{\mu}_i \pd \tilde{\mu}_j} d\tilde{\mu}_j
+ \sqrt{-1} d\tilde{\theta}_i = \frac{d\tilde{z}_i}{\tilde{z}_i}.
\ee
By a direct computation,
we have
\bea
\frac{d\tilde{z}_1}{\tilde{z}_1} & = & \frac{dz_1}{z_1} - \frac{dz_2}{z_2}, \\
\frac{d\tilde{z}_2}{\tilde{z}_2} & = & \frac{dz_2}{z_2}.
\eea
Therefore, we may take
\begin{align}
 \tilde{z}_1 & = \frac{z_1}{z_2}, & \tilde{z}_2 & = z_2.
\end{align}
By \eqref{eqn:Alpha-Beta},
the relationship between the $(\tilde{z}_1, \tilde{z}_2)$-coordinates
and the $(\alpha, \beta)$-coordinates is given by:
\be \label{eqn:Alpha-Beta-tilde}
\tilde{z}_1 = \alpha^{-1},  \tilde{z}_2 = \alpha\beta.
\ee
Note we have
\begin{align}
\mu_2 & = \half |\tilde{z}_1|^{-2}, &
z & = \frac{1}{2}(|\tilde{z}_1|^{-2}-|\tilde{z}_1|^2|\tilde{z}_2|^2).
\end{align}
Therefore, when $\tilde{z}_2 =0$,
\be
\mu_2 = -\mu_1 = z = |\tilde{z}_1|^{-2} \geq 0.
\ee

\subsection{The $n=2$ case}

Now we generalize the results in the case of $n=1$ to $n=2$.
From the equations
\begin{align}
z_1 & = \prod_{j=1}^2 (r_j-(z-c_j))^{1/2}\cdot e^{\sqrt{-1} \theta_1},
& z_2 & = x + \sqrt{-1} y
\end{align}
we solve for $z$ and get two candidate solutions:
\ben
&& z - c_1 = \half (c_2-c_1) \pm \frac{|z_1|^2-|z_2|^2}{2|z_1|(|z_1|^2+|z_2|^2)}
\sqrt{(|z_1|^2+|z_2|^2)^2+(c_2-c_1)^2|z_1|^2}.
\een
One can fix one solution by requiring that
as $x\to 0$ and $y\to 0$, $z$ tends to a  number $<c_1$,
since we are working over the domain $U = \bR^3 - \{(0,0,z)\;|\; z \geq c_1\}$,
so we have
\be \label{eqn:z-z1z2}
z - c_1 = \half (c_2-c_1) - \frac{|z_1|^2-|z_2|^2}{2|z_1|(|z_1|^2+|z_2|^2)} \sqrt{R}
\ee
where $R$ is defined by:
\be \label{eqn:R}
R:= (|z_1|^2+|z_2|^2)^2+(c_2-c_1)^2|z_1|^2.
\ee
From this one can check that:
\bea
&& r_1 = (c_1-c_2) \frac{|z_1|^2-|z_2|^2}{2(|z_1|^2+|z_2|^2)}
+ \frac{1}{2|z_1|} \sqrt{R}, \label{eqn:r1}  \\
&& r_2 = (c_2-c_1) \frac{|z_1|^2-|z_2|^2}{2(|z_1|^2+|z_2|^2)}
+ \frac{1}{2|z_1|} \sqrt{R}. \label{eqn:r2}
\eea
It follows that
\ben
&& \frac{1}{r_1} = \frac{2((|z_1|^4-|z_2|^4)|z_1|^2c+|z_1|(|z_1|^2+|z_2|^2)^2\sqrt{R} )}
{(|z_1|^2+|z_2|^2)^4+4|z_1|^4|z_2|^2c^2}, \\
&& \frac{1}{r_2} = \frac{2(-(|z_1|^4-|z_2|^4)|z_1|^2c+|z_1|(|z_1|^2+|z_2|^2)^2\sqrt{R} )}
{(|z_1|^2+|z_2|^2)^4+4|z_1|^4|z_2|^2c^2}.
\een
So we have:
\be
V = \frac{1}{2r_1} +\frac{1}{2r_2} = \frac{2|z_1|(|z_1|^2+|z_2|^2)^2\sqrt{R}}
{(|z_1|^2+|z_2|^2)^4+4|z_1|^4|z_2|^2c^2}.
\ee
We also have
\ben
&& r_1 +z-c_1 =   \frac{(c_2-c_1)|z_2|^2}{|z_1|^2+|z_2|^2}
+ \frac{|z_2|^2}{|z_1|(|z_1|^2+|z_2|^2)} \sqrt{R},   \\
&& r_2 +z-c_2 =   \frac{(c_1-c_2)|z_2|^2}{|z_1|^2+|z_2|^2}
+ \frac{|z_2|^2}{|z_1|(|z_1|^2+|z_2|^2)} \sqrt{R},   \\
&& \sum_{j=1}^2 \frac{r_j +z-c_j}{r_j}
= \frac{4|z_2|^2 [(|z_1|^2+|z_2|^2)^3+2|z_1|^4c^2]}{(|z_1|^2+|z_2|^2)^4+4|z_1|^4|z_2|^2c^2}.
\een
Therefore,
\ben
&& V \rho^2 + \frac{1}{4V}
\biggl( \sum_{j=1}^n \frac{r_j+z-c_j}{r_j}\biggr)^2 \\
& = & V^2\rho^2 + \frac{1}{4}
\biggl(\frac{4|z_2|^2 [(|z_1|^2+|z_2|^2)^3+2|z_1|^4c^2]}{(|z_1|^2+|z_2|^2)^4+4|z_1|^4|z_2|^2c^2} \biggr)^2\\
& = & |z_2|^2
\biggl(\frac{2|z_1|(|z_1|^2+|z_2|^2)^2\sqrt{R}}
{(|z_1|^2+|z_2|^2)^4+4|z_1|^4|z_2|^2c^2}\biggr)^2 \\
&& + \frac{1}{4} \biggl(\frac{4|z_2|^2 [(|z_1|^2+|z_2|^2)^3+2|z_1|^4c^2]}{(|z_1|^2+|z_2|^2)^4+4|z_1|^4|z_2|^2c^2}\biggr)^2 \\
& = & \frac{4|z_2|^2[(|z_1|^2+|z_2|^2)^3+c^2|z_1|^4]}{(|z_1|^2+|z_2|^2)^4+4|z_1|^4|z_2|^2c^2}.
\een

By \eqref{eqn:g-in-complex} we then have:
\ben
g & = & \frac{(|z_1|^2+|z_2|^2)^4+4|z_1|^4|z_2|^2c^2}{2|z_1|(|z_1|^2+|z_2|^2)^2\sqrt{R}}
\frac{dz_1}{z_1}  \frac{d \bar{z}_1}{\bar{z}_1} \\
& - & \frac{|z_2|^2 [(|z_1|^2+|z_2|^2)^3+2|z_1|^4c^2]}{|z_1|(|z_1|^2+|z_2|^2)^2\sqrt{R}}
\biggl( \frac{dz_1}{z_1} \frac{d\bar{z}_2}{\bar{z}_2}
+ \frac{dz_2}{z_2} \frac{d\bar{z}_1}{\bar{z}_1} \biggr) \\
& + &
\frac{2|z_2|^2[(|z_1|^2+|z_2|^2)^3+c^2|z_1|^4]}{|z_1|(|z_1|^2+|z_2|^2)^2\sqrt{R}}
\frac{dz_2}{z_2} \frac{d\bar{z}_2}{\bar{z}_2}.
\een

\subsubsection{The $(\alpha, \beta)$-coordinates}

As in the $n=1$ case, make the following change of variables:
\be \label{eqn:Alpha-Beta2}
z_1 = \beta,  z_2 = \alpha\beta.
\ee
The Gibbons-Hawking metric becomes:
\ben
g & = &  \frac{|\beta|^2(1+|\alpha|^2)^4+4|\alpha|^2c^2}
{2(1+|\alpha|^2)^2\sqrt{|\beta|^2(1+|\alpha|^2)+c^2}}
\frac{d\beta}{\beta}  \frac{d \bar{\beta}}{\bar{\beta}} \\
& - & \frac{|\alpha|^2(|\beta|^2(1+|\alpha|^2)^3+2c^2)}
{(1+|\alpha|^2)^2\sqrt{|\beta|^2(1+|\alpha|^2)^2+c^2}}
\biggl( \frac{d\beta}{\beta}
\biggl(\frac{d\bar{\alpha}}{\bar{\alpha}}
+  \frac{d\bar{\beta}}{\bar{\beta}}  \biggr)
+ \biggl(\frac{d\alpha}{\alpha}
+ \frac{d\beta}{\beta}  \biggr)
\frac{d\bar{\beta}}{\bar{\beta}} \biggr) \\
& + &
\frac{2|\alpha|^2(|\beta|^2(1+|\alpha|^2)^3+c^2)}{(1+|\alpha|^2)^2\sqrt{|\beta|^2(1+|\alpha|^2)^2+c^2}}
\biggl(\frac{d\alpha}{\alpha} + \frac{d\beta}{\beta}  \biggr)
\biggl(\frac{d\bar{\alpha}}{\bar{\alpha}}
+  \frac{d\bar{\beta}}{\bar{\beta}}  \biggr) \\
& = & \frac{2|\alpha|^2(|\beta|^2(1+|\alpha|^2)^3+c^2)}{(1+|\alpha|^2)^2\sqrt{|\beta|^2(1+|\alpha|^2)^2+c^2}}
d\alpha d\bar{\alpha} \\
& + &\frac{|\alpha|^2|\beta|^2(1+|\alpha|^2)}{\sqrt{|\beta|^2(1+|\alpha|^2)^2+c^2}}
\biggl(\frac{d\alpha}{\alpha} \frac{d\bar{\beta}}{\bar{\beta}}
+ \frac{d\beta}{\beta} \frac{d\bar{\alpha}}{\bar{\alpha}} \biggr) \\
& + &  \frac{(1+|\alpha|^2)^2}{2\sqrt{|\beta|^2(1+|\alpha|^2)^2+c^2}}
 d\beta d\bar{\beta}.
\een
And so the K\"ahler form becomes:
\ben
\omega
& = & \frac{1}{2\sqrt{-1}} \biggl\{\frac{2|\alpha|^2(|\beta|^2(1+|\alpha|^2)^3+c^2)}{(1+|\alpha|^2)^2\sqrt{|\beta|^2(1+|\alpha|^2)^2+c^2}}
d\alpha \wedge d\bar{\alpha} \\
& + &\frac{|\alpha|^2|\beta|^2(1+|\alpha|^2)}{\sqrt{|\beta|^2(1+|\alpha|^2)^2+c^2}}
\biggl(\frac{d\alpha}{\alpha} \wedge \frac{d\bar{\beta}}{\bar{\beta}}
+ \frac{d\beta}{\beta}  \wedge \frac{d\bar{\alpha}}{\bar{\alpha}} \biggr) \\
& + &  \frac{(1+|\alpha|^2)^2}{2\sqrt{|\beta|^2(1+|\alpha|^2)^2+c^2}}
 d\beta \wedge d\bar{\beta} \biggr\}.
\een
This matches with \cite{Duan-Zhou1} and \cite{Zhou-Kepler}.
The K\"ahler potential is
\be \label{eqn:Kahler-pot-n=2}
\psi^\vee = c \log \frac{\sqrt{|\beta|^2(1+|\alpha|^2)^2+c^2}-c}{1+|\alpha|^2}
+ \sqrt{|\beta|^2(1+|\alpha|^2)^2+c^2}.
\ee
In the next subsection we will use this to fix the constants in \eqref{eqn:Kahler}.

\subsubsection{Moment map in $(\alpha, \beta)$-coordinates}

We express many quantities in terms of these new variables:
\bea
&& z  =  \half (c_1+c_2) - \frac{1-|\alpha|^2}{2(1+|\alpha|^2)}
\sqrt{|\beta|^2(1+|\alpha|^2)^2+(c_2-c_1)^2}, \\
&& r_1 = (c_1-c_2) \frac{1-|\alpha|^2}{2(1+|\alpha|^2)} + \half
\sqrt{|\beta|^2(1+|\alpha|^2)^2+(c_2-c_1)^2}, \\
&& r_2 = -(c_1-c_2) \frac{1-|\alpha|^2}{2(1+|\alpha|^2)} + \half
\sqrt{|\beta|^2(1+|\alpha|^2)^2+(c_2-c_1)^2}.
\eea
From these identities we easily get:
\ben
&& \mu_2  = \half\sum_{j=1}^2 (r_j+z-c_j)
= \frac{|\alpha|^2}{1+|\alpha|^2} \sqrt{|\beta|^2(1+|\alpha|^2)^2+c^2},  \\
&& \mu_2-(z-c_1) = - \frac{c}{2} + \half \sqrt{|\beta|^2(1+|\alpha|^2)^2+c^2},  \\
&& \mu_2-(z-c_2) =  + \frac{c}{2} + \half \sqrt{|\beta|^2(1+|\alpha|^2)^2+c^2}, \\
&& \mu_2-(z-c_1)-(z-c_2) = \frac{1}{1+|\alpha|^2} \sqrt{|\beta|^2(1+|\alpha|^2)^2+c^2}, \\
&& 2\mu_2-(z-c_1)-(z-c_2) = \sqrt{|\beta|^2(1+|\alpha|^2)^2+c^2},
\een
where $c$ is defined by $c:=c_2-c_1$.
It follows that  the image of moment map is given by
\begin{align}
\mu_2 & \geq 0, & \mu_2 + \mu_1 + c_1 & \geq 0, &
\mu_2 + (\mu_1+c_1) + (\mu_1+c_2) & \geq 0.
\end{align}
It is convex domain whose boundary consists of  three linear pieces:
\ben
&& L_1 = \{(\mu_1, 0)\;|\; \mu_1 \geq -c_1\}, \\
&& L_2 = \{(\mu_1, -\mu_1-c_1)\;|;\ -c_2 \leq \mu_1 \leq -c_1\}, \\
&& L_3 = \{(\mu_1, -2\mu_1-c_1-c_2)\;|\; \mu_1 \leq -c_2\}.
\een
Note when $\alpha = 0$,
\ben
&& z|_{\alpha = 0}  = \half (c_1+c_2) - \frac{1}{2}
\sqrt{|\beta|^2+(c_2-c_1)^2} \in (-\infty, c_1], \\
&& \mu_2|_{\alpha=0}= 0,
\een
and so the $\beta$-axix in the $(\alpha,\beta)$-plane is mapped to $L_1$.
Similarly,
\ben
&& z|_{\beta = 0} = \half (c_1+c_2) - \frac{1-|\alpha|^2}{2(1+|\alpha|^2)} (c_2-c_1)
\in [c_1, c_2), \\
&& \mu_2|_{\beta=0} = \frac{|\alpha|^2}{1+|\alpha|^2} (c_2-c_1),
\een
and so
\ben
(\mu_1+\mu_2)|_{\beta = 0} = -c_1,
\een
and so the $\alpha$-axis in the $(\alpha,\beta)$-plane is mapped to $L_2$.

Note now we have
\bea
&& r_1-(z-c_1) = \frac{\sqrt{|\beta|^2(1+|\alpha|^2)^2+c^2}-c}{1+|\alpha|^2}, \\
&& r_2-(z-c_2) = \frac{\sqrt{|\beta|^2(1+|\alpha|^2)^2+c^2}+c}{1+|\alpha|^2}.
\eea
Therefore, by \eqref{eqn:Kahler},
\ben
\psi^\vee& = & - c_1 \log \frac{\sqrt{|\beta|^2(1+|\alpha|^2)^2+c^2}-c}{1+|\alpha|^2}
- c_2 \log \frac{\sqrt{|\beta|^2(1+|\alpha|^2)^2+c^2}+c}{1+|\alpha|^2}\\
& - & C_1\biggl(\half (c_1+c_2) - \frac{1-|\alpha|^2}{2(1+|\alpha|^2)}
\sqrt{|\beta|^2(1+|\alpha|^2)^2+(c_2-c_1)^2} \biggr) \\
& + & C_2\frac{|\alpha|^2}{1+|\alpha|^2} \sqrt{|\beta|^2(1+|\alpha|^2)^2+c^2},
\een
and so by comparing with \eqref{eqn:Kahler-pot-n=2},
$C_1=C_2=2$ in  \eqref{eqn:Kahler}.

\subsubsection{Change of Hessian local complex coordinates in the $n=2$ case}

In the above we have focused on the local coordinate patch over $U = \bR^3 -
\{ (0, 0, z) \;|\; z \geq c_1 \}$,
now we switch the local coordinate patch over $\tilde{U} = \bR^3 - \{ (0,0,z) \;|\; z \leq c_2\}$
where we have
\ben
&& g = \frac{1}{V}(d\tilde{\varphi} + \tilde{\alpha})^2 + V \cdot (dx^2+dy^2+dz^2).
\een
By \eqref{eqn:tildeAlpha-Alpha-n}, $\tilde{\alpha} = \alpha + 2 \theta$,
so we can take
\be
\tilde{\varphi} = \varphi -2 \arctan\frac{y}{x} = \varphi - 2 \theta_2,
\ee
and so as in the $n=1$ case,
\begin{align}
\tilde{\theta}_1 & = \tilde{\varphi} = \theta_1-2\theta_2,
& \tilde{\mu}_1 & = \mu_1 = - z, \\
\tilde{\theta}_2 & = \theta_2 =  \arctan \frac{y}{x}, &
\tilde{\mu}_2 & = \mu_2 + 2 z =  \half\sum_{j=1}^n (r_j+z-c_j) +2 z.
\end{align}
The complex potential becomes
\be
\psi = (\tilde{\mu}_2+2\tilde{\mu}_1)\log (\tilde{\mu}_2+2\tilde{\mu}_1)
+ \tilde{\mu}_2 \log\tilde{\mu}_2.
\ee
From the complex potential
 one can find local complex coordinates $\tilde{z}_1$ and $\tilde{z}_2$ so that
\be
\half \sum_{j=1}^2 \frac{\pd^2\psi}{\pd \tilde{\mu}_i \pd \tilde{\mu}_j} d\tilde{\mu}_j
+ \sqrt{-1} d\tilde{\theta}_i = \frac{d\tilde{z}_i}{\tilde{z}_i}.
\ee
By a direct computation,
we have
\bea
\frac{d\tilde{z}_1}{\tilde{z}_1} & = & \frac{dz_1}{z_1} - 2\frac{dz_2}{z_2}, \\
\frac{d\tilde{z}_2}{\tilde{z}_2} & = & \frac{dz_2}{z_2}.
\eea
Therefore, we may take
\begin{align}
 \tilde{z}_1 & = \frac{z_1}{z_2^2}, & \tilde{z}_2 & = z_2.
\end{align}
By \eqref{eqn:Alpha-Beta-tilde},
introduce $(\tilde{\alpha}, \tilde{\beta})$-coordinates by:
\be \label{eqn:Alpha-Beta-tilde2}
\tilde{z}_1 = \tilde{\alpha}^{-1},  \tilde{z}_2 = \tilde{\alpha} \tilde{\beta}.
\ee
Together with \eqref{eqn:Alpha-Beta2},
we derive the following relations:
\begin{align}
\tilde{\alpha} & = \alpha^2\beta, & \tilde{\beta} & = \frac{1}{\alpha}.
\end{align}
In the $(\tilde{\alpha}, \tilde{\beta})$-coordinates
\ben
&& \mu_2 = \frac{1}{|\tilde{\beta}|^2+1} \sqrt{|\tilde{\alpha}|^2(|\tilde{\beta}|^2+1)^2+c^2}, \\
&& z= \half (c_1+c_2)
- \frac{|\tilde{\beta}|^2-1}{2(|\tilde{\beta}|^2+1)}
\sqrt{|\tilde{\alpha}|^2(|\tilde{\beta}|^2+1)^2+c^2}.
\een
When $\tilde{\alpha} = 0$,
\begin{align*}
\mu_2 & = \frac{c}{|\tilde{\beta}|^2+1}, &
z & = \frac{c_1+c_2}{2} - \frac{(|\tilde{\beta}|^2-1)c}{2(|\tilde{\beta}|^2+1)},
\end{align*}
and therefore $\mu_2+\mu_1 = - c_1$,
and the $\tilde{\beta}$-axis in the $(\tilde{\alpha}, \tilde{|beta})$-plane
is mapped by the moment map to the $L_2$ part of the boundary.
When $\tilde{\beta} =0$,
\begin{align*}
\mu_2 & = \sqrt{|\tilde{\alpha}|^2+c^2}, &
z & = \half (c_1+c_2) + \frac{1}{2} \sqrt{|\tilde{\alpha}|^2+c^2},
\end{align*}
and so $\mu_2 + 2 \mu_1 = - (c_1+c_2)$,
therefore,
the $\tilde{\alpha}$-axis in the $(\tilde{\alpha}, \tilde{\beta})$-plane
is mapped to the $L_3$ part of the boundary.

There is another local coordinate patch,
the one over $\hat{U} = \bR^3 - \{ (0,0,z) \;|\; z \leq c_1\}$
where we have
\ben
&& g = \frac{1}{V}(d\hat{\varphi} + \hat{\alpha})^2 + V \cdot (dx^2+dy^2+dz^2),
\een
where $\hat{\alpha} = \alpha + \theta$,
so we can take
\ben
\hat{\varphi} = \varphi - \arctan\frac{y}{x} = \varphi - \theta_2,
\een
and therefore take the following symplectic coordinates,
\begin{align*}
\hat{\theta}_1 & = \hat{\varphi} = \theta_1-\theta_2,
& \hat{\mu}_1 & = \mu_1 = - z, \\
\hat{\theta}_2 & = \theta_2 =  \arctan \frac{y}{x}, &
\hat{\mu}_2 & = \mu_2 + z =  \half\sum_{j=1}^n (r_j+z-c_j) + z.
\end{align*}
Therefore, we may take
\begin{align}
 \hat{z}_1 & = \frac{z_1}{z_2}, & \hat{z}_2 & = z_2.
\end{align}
By \eqref{eqn:Alpha-Beta},
introduce $(\hat{\alpha}, \hat{\beta})$-coordinates by:
\be \label{eqn:Alpha-Beta-hat}
\hat{z}_1 = \hat{\beta},  \hat{z}_2 = \hat{\alpha} \hat{\beta}.
\ee
Together with \eqref{eqn:Alpha-Beta2},
we derive the following relations:
\begin{align}
\hat{\alpha} & = \alpha^2\beta, & \hat{\beta} & = \frac{1}{\alpha}.
\end{align}
It follows that we have
\begin{align}
\hat{\alpha} & = \tilde{\alpha}, &
\hat{\beta} & = \tilde{\beta}.
\end{align}

\subsection{Hessian local complex coordinates on toric Gibbons-Hawking spaces}

\label{sec:General}

Generalizing the $n=1$ and $n=2$ case,
one sees that for general $n$,
the toric Gibbons-Hawking space is covered by $n$ local coordinate patches,
with local coordinates $(\alpha_j, \beta_j)$, $j=1, \dots, n$.
The coordinates $(\alpha_1, \beta_1)$ are given by:
\begin{align}
\alpha_1 & = \frac{z_2}{z_1}, &
\beta_1 & = z_1,
\end{align}
where $z_1, z_2$ are defined by:
\bea
&& z_1 = \prod_{j=1}^n (r_j-(z-c_j))^{1/2}\cdot e^{\sqrt{-1} \theta_1}, \\
&& z_2 = x+ \sqrt{-1} y.
\eea
And for $i=1, \dots, n-1$,
\begin{align} \label{eqn:Coordinate change}
\alpha_{i+1} & = \alpha_i^2 \beta_i, & \beta_{i+1} & = \alpha_i^{-1},
\end{align}
In other words,
the toric Gibbons-Hawking space can be obtained by gluing $n$ copies of $\bC^2$ as follows.
For $i=1, \dots, n$, denote by $U_i$ the $i$-th copy of $\bC^2$
and let $(\alpha_i, \beta_i)$ be the linear coordinates on it.
Then $U_i$ and $U_{i+1}$ are glued together
by the above formula for change of coordinates.
This space is nothing but the toric crepant resolution $\widehat{\bC^2/\bZ_n}$
of the orbifold $\bC^2/\bZ_n$.
Therefore,
we can summarize our discussions so far as follows: Given $n$ distinct real numbers $c_1 < \cdots < c_n$,
one obtains via the Gibbons-Hawking construction a K\"ahler Ricci-flat metric
on $X_n:=\widehat{\bC^2/\bZ_n}$,
whose K\"ahler potential is given by:
\be
\psi^\vee= -\sum_{j=1}^n c_j \log(r_j-(z-c_j))+C_1\mu_1 + C_2\mu_2
\ee
for some constants $C_1$ and $C_2$.
Here $\mu_1, \mu_2$ are the two components of the moment map for  Hamiltonian $2$-torus
action on $X_n$,
\begin{align}
\mu_1 & = - z, & \mu_2 = \half \sum_{i=1}^n (r_i - (z-c_j)).
\end{align}
By \eqref{eqn:action},
the torus action is given by:
\be
(e^{it_1}, e^{it_2}) \cdot (z_1, z_2) = (e^{it_1} z_1, e^{it_2}z_2).
\ee
in the $(z_1, z_2)$-coordinates,
hence it is given by
\be \label{eqn:Action2}
(e^{it_1}, e^{it_2}) \cdot (\alpha_1, \beta_1) = (e^{i(t_2-t_1)} \alpha_1, e^{it_1}\beta_1).
\ee
in the $(\alpha_1, \beta_1)$-coordinates.
The image of the moment map is a convex body whose boundary consists
of consecutively  a ray $L_1$, $n-1$ intervals $L_2, \dots, L_{n}$, and another
ray $L_{n+1}$.
Furthermore,
the $\alpha_i$-axis in the $(\alpha_i, \beta_i)$-plane is mapped to $L_{i+1}$,
and the $\beta_i$-axis is mapped to $L_i$, for $i=1, \dots, n$.

\section{Phase Changes of Toric Gibbons-Hawking Metrics}

\label{sec:Phase}

In this Section we discuss the phase change \cite{Duan-Zhou1, Duan-Zhou2} of Gibbons-Hawking metrics
as applications of the results in preceding Sections.

\subsection{A phase change of Eguchi-Hanson metrics}

Let us now focus on the case of $n=2$
and discuss the phase change phenomena of the K\"ahler Ricci-flat metric
introduced in \cite{Duan-Zhou1, Duan-Zhou2}.
We have seen that in this case the K\"ahler potential is given by
\be
\psi_c^\vee = c \log \frac{\sqrt{|\beta|^2(1+|\alpha|^2)^2+c^2}-c}{1+|\alpha|^2}
+ \sqrt{|\beta|^2(1+|\alpha|^2)^2+c^2},
\ee
and the K\"ahler form is
\ben
\omega_c
& = & \sqrt{-1} \biggl\{ \frac{1}{(1+|\alpha|^2)^2}
\biggl( \sqrt{|\beta|^2(1+|\alpha|^2)^2+c^2} + \frac{|\beta|^2(1+|\alpha|^2)^2|z|^2}
{\sqrt{|\beta|^2(1+|\alpha|^2)^2+c^2}}
\biggr) d\alpha \wedge d \bar{\alpha} \\
& + &  \frac{(1+|\alpha|^2)}{2\sqrt{|\beta|^2(1+|\alpha|^2)^2 + c^2}} \bar{\alpha} \beta
\cdot d\alpha  \wedge d\bar{\beta}
  +  \frac{(1+|\alpha|^2)}{2\sqrt{|\beta|^2(1+|\alpha|^2)^2 + c^2}} \alpha \bar{\beta}
\cdot d\beta  \wedge d \bar{\alpha} \\
& + &  \frac{(1+|\alpha|^2)^2}{4\sqrt{|\beta|^2(1+|\alpha|^2)^2 + c^2}} \cdot d\beta  \wedge d \bar{\beta}
\biggr\}.
\een
With respect to the action \eqref{eqn:Action2},
we can take
\begin{align}
\mu_1^c & = \frac{1-|\alpha|^2}{2(1+|\alpha|^2)}
\sqrt{|\beta|^2(1+|\alpha|^2)^2+c^2}, \\
\mu_2^c & = \frac{|\alpha|^2}{1+|\alpha|^2} \sqrt{|\beta|^2(1+|\alpha|^2)^2+c^2}.
\end{align}
The image of the moment map is the convex body defined by:
\begin{align}
\mu_2^c & \geq 0, & \mu_2^c+\mu_1^c & \geq \frac{c}{2}, & \mu_2^c + 2\mu_1^c & \geq 0.
\end{align}

Now we change $c$ to $b\sqrt{-1}$ for some $b > 0$,
then the K\"ahler potential becomes:
\be
\psi_{b\sqrt{-1}}^\vee = \sqrt{-1} a \log \frac{\sqrt{|\beta|^2(1+|\alpha|^2)^2-b^2}- b \sqrt{-1}}{1+|\alpha|^2}
+ \sqrt{|\beta|^2(1+|\alpha|^2)^2-b^2}.
\ee
We have two cases to consider.

\subsubsection{Case 1}

When $|\beta|^2(1+|\alpha|^2)^2-b^2 \geq 0$,
one can rewrite the K\"ahler potential as:
\be
\psi_{b\sqrt{-1}}^\vee = \sqrt{|\beta|^2(1+|\alpha|^2)^2-b^2}
- b \arctan \frac{\sqrt{|\beta|^2(1+|\alpha|^2)^2-b^2}}{b} + \frac{\pi}{2}.
\ee
This is still a purely real-valued function.
Then one has
\ben
\frac{\pd \psi_{b\sqrt{-1}}^\vee}{\pd \bar{\alpha}}
& = & \frac{|\beta|^2(1+|\alpha|^2)\alpha}{\sqrt{|\beta|^2(1+|\alpha|^2)^2-b^2}}
- b \frac{\frac{|\beta|^2(1+|\alpha|^2)\alpha}{a\sqrt{|\beta|^2(1+|\alpha|^2)^2-b^2}}}
{1 + \frac{|\beta|^2(1+|\alpha|^2)^2-b^2}{a^2} } \\
& = & \frac{\alpha}{1+|\alpha|^2}\sqrt{|\beta|^2(1+|\alpha|^2)^2-b^2},\\
\frac{\pd \psi_{b\sqrt{-1}}^\vee}{\pd \bar{\beta}}
& = & \frac{(1+|\alpha|^2)^2\beta}{2\sqrt{|\beta|^2(1+|\alpha|^2)^2-b^2}}
- b \frac{\frac{(1+|\alpha|^2)^2\beta}{2a\sqrt{|\beta|^2(1+|\alpha|^2)^2-b^2}}}
{1 + \frac{|\beta|^2(1+|\alpha|^2)^2-b^2}{b^2} } \\
& = & \frac{1}{2\bar{\beta}}\sqrt{|\beta|^2(1+|\alpha|^2)^2-b^2},
\een
and the second derivatives are:
\ben
\frac{\pd^2 \psi_{b\sqrt{-1}}^\vee}{\pd \alpha \pd \bar{\alpha}}
& = & \frac{\sqrt{|\beta|^2(1+|\alpha|^2)^2-b^2}}{(1+|\alpha|^2)^2}
+ \frac{|\alpha|^2|\beta|^2}{\sqrt{|\beta|^2(1+|\alpha|^2)^2-b^2}}, \\
\frac{\pd^2 \psi_{b\sqrt{-1}}^\vee}{\pd \alpha \pd \bar{\beta}}
& = & \frac{(1+|\alpha|^2)\bar{\alpha}\beta}{2\sqrt{|\beta|^2(1+|\alpha|^2)^2-b^2}}, \\
\frac{\pd^2 \psi_{b\sqrt{-1}}^\vee}{\pd \beta \pd \bar{\beta}}
& = & \frac{(1+|\alpha|^2)^2}{4\sqrt{|\beta|^2(1+|\alpha|^2)^2-b^2}}.
\een
So in this case the K\"ahler form becomes:
\ben
\omega_{b\sqrt{-1}}
& = & \sqrt{-1} \biggl\{ \frac{1}{(1+|\alpha|^2)^2}
\biggl( \sqrt{|\beta|^2(1+|\alpha|^2)^2-b^2} + \frac{|\beta|^2(1+|\alpha|^2)^2|z|^2}
{\sqrt{|\beta|^2(1+|\alpha|^2)^2-b^2}}
\biggr) d\alpha \wedge d \bar{\alpha} \\
& + &  \frac{(1+|\alpha|^2)}{2\sqrt{|\beta|^2(1+|\alpha|^2)^2 - b^2}} \bar{\alpha} \beta
\cdot d\alpha  \wedge d\bar{\beta}
  +  \frac{(1+|\alpha|^2)}{2\sqrt{|\beta|^2(1+|\alpha|^2)^2 - b^2}} \alpha \bar{\beta}
\cdot d\beta  \wedge d \bar{\alpha} \\
& + &  \frac{(1+|\alpha|^2)^2}{4\sqrt{|\beta|^2(1+|\alpha|^2)^2 - b^2}} \cdot d\beta  \wedge d \bar{\beta}
\biggr\}.
\een
With respect to the action \eqref{eqn:Action2},
we can take
\begin{align}
\mu_1^{b\sqrt{-1}} & = \frac{1-|\alpha|^2}{2(1+|\alpha|^2)}
\sqrt{|\beta|^2(1+|\alpha|^2)^2-b^2}, \\
\mu_2^{b\sqrt{-1}} & = \frac{|\alpha|^2}{1+|\alpha|^2} \sqrt{|\beta|^2(1+|\alpha|^2)^2-b^2}.
\end{align}
The image of the moment map is the convex cone defined by:
\begin{align}
\mu_2^{b\sqrt{-1}} & \geq 0, & \mu_2^{b\sqrt{-1}}+\mu_1^{b\sqrt{-1}} & \geq 0, &
\mu_2^{b\sqrt{-1}} + 2\mu_1^{b\sqrt{-1}} & \geq 0.
\end{align}
In fact, in this case
the second condition is implied by the first and the third conditions.

\subsubsection{Case 2}
When $|\beta|^2(1+|\alpha|^2)^2-b^2 < 0$,
one can rewrite the K\"ahler potential as:
\be
\begin{split}
\psi_{b\sqrt{-1}}^\vee = \sqrt{-1} \biggl( & b \log \frac{b-\sqrt{b^2-|\beta|^2(1+|\alpha|^2)^2}}{1+|\alpha|^2} \\
& + \sqrt{b^2-|\beta|^2(1+|\alpha|^2)^2} + \frac{3\pi \sqrt{-1}b}{2} \biggr).
\end{split}
\ee
Up to an unimportant constant,
$\psi^\vee$ becomes purely imaginary,
hence it does not define a pseudo-K\"ahler metric,
but instead, a {\em purely imaginary pseudo-K\"ahler metric}.
Its K\"ahler form can be found by taking its derivatives:
\ben
\frac{\pd\psi^\vee_{b\sqrt{-1}}}{\pd \bar{\alpha}}
& = & \sqrt{-1} \biggl( \frac{-|\beta|^2(1+|\alpha|^2)\alpha}{\sqrt{b^2-|\beta|^2(1+|\alpha|^2)^2}}
- \frac{b\frac{-|\beta|^2(1+|\alpha|^2)\alpha}{\sqrt{b^2-|\beta|^2(1+|\alpha|^2)^2}}}{b+\sqrt{a^2-|\beta|^2(1+|\alpha|^2)^2}}
 + \frac{b\alpha}{1+ |\alpha|^2} \biggr) \\
& = & \sqrt{-1} \biggl( \frac{-|\beta|^2(1+|\alpha|^2)\alpha}{b+ \sqrt{b^2-|\beta|^2(1+|\alpha|^2)^2}}
 + \frac{b\alpha}{1+ |\alpha|^2} \biggr) \\
& = &  \sqrt{-1} \biggl( \frac{\alpha}{1+|\alpha|^2}\sqrt{b^2-|\beta|^2(1+|\alpha|^2)^2} \biggr).
\een

\ben
\frac{\pd\psi^\vee_{b\sqrt{-1}}}{\pd \bar{\beta}}
& = & \sqrt{-1} \biggl( \frac{-\beta(1+|\alpha|^2)^2}{2\sqrt{b^2-|\beta|^2(1+|\alpha|^2)^2}}
- \frac{b\frac{-\beta(1+|\alpha|^2)^2}{2\sqrt{b^2-|\beta|^2(1+|\alpha|^2)^2}}}{b+\sqrt{b^2-|\beta|^2(1+|\alpha|^2)^2}}  \biggr) \\
& = & \sqrt{-1} \biggl(\frac{-\beta(1+|\alpha|^2)^2}{2(b+\sqrt{b^2-|\beta|^2(1+|\alpha|^2)^2})}
= -\frac{b}{2\bar{\beta}} + \frac{\sqrt{b^2-|\beta|^2(1+|\alpha|^2)^2}}{2\bar{\beta}}  \biggr).
\een

\ben
\frac{\pd^2\psi^\vee_{b\sqrt{-1}}}{\pd \alpha\pd \bar{\alpha}}
& = & \sqrt{-1} \biggl(\frac{\sqrt{b^2-|\beta|^2(1+|\alpha|^2)^2}}{(1+|\alpha|^2)^2}
- \frac{|\alpha|^2|\beta|^2}{\sqrt{b^2-|\beta|^2(1+|\alpha|^2)^2}}  \biggr).
\een

\ben
\frac{\pd^2\psi^\vee_{b\sqrt{-1}}}{\pd \alpha\pd \bar{\beta}}
& = & \sqrt{-1} \biggl( -
\frac{(1+|\alpha|^2)\bar{\alpha}\beta}{2\sqrt{b^2-|\beta|^2(1+|\alpha|^2)^2}}  \biggr).
\een

\ben
\frac{\pd\psi^\vee_{b\sqrt{-1}}}{\pd \beta\pd \bar{\beta}}
& = & \sqrt{-1} \biggl(-\frac{(1+|\alpha|^2)^2}{4\sqrt{b^2-|\beta|^2(1+|\alpha|^2)^2}}  \biggr) .
\een
And so the K\"ahler form is
\ben
\omega_{b\sqrt{-1}}
& = & \biggl(- \frac{\sqrt{b^2-|\beta|^2(1+|\alpha|^2)^2}}{(1+|\alpha|^2)^2}
+ \frac{|\alpha|^2|\beta|^2}{\sqrt{b^2-|\beta|^2(1+|\alpha|^2)^2}}  \biggr) d \alpha \wedge d\bar{\alpha} \\
& + & \frac{(1+|\alpha|^2)\bar{\alpha}\beta}{2\sqrt{b^2-|\beta|^2(1+|\alpha|^2)^2}}
d \alpha \wedge d\bar{\beta}
+ \frac{(1+|\alpha|^2)\alpha\bar{\beta}}{2\sqrt{b^2-|\beta|^2(1+|\alpha|^2)^2}}
d \beta \wedge d\bar{\alpha} \\
& + & \frac{(1+|\alpha|^2)^2}{4\sqrt{b^2-|\beta|^2(1+|\alpha|^2)^2}}
d \beta \wedge d\bar{\beta}.
\een
Note this is a {\em purely imaginary form}.
With respect to the action \eqref{eqn:Action2},
we can take the moment map to be given by the purely imaginary-valued functions:
\begin{align}
\mu_1^{b\sqrt{-1}} & = \sqrt{-1} \frac{1-|\alpha|^2}{2(1+|\alpha|^2)}
\sqrt{b^2-|\beta|^2(1+|\alpha|^2)^2}, \\
\mu_2^{b\sqrt{-1}} & = \sqrt{-1}\frac{|\alpha|^2}{1+|\alpha|^2} \sqrt{b^2-|\beta|^2(1+|\alpha|^2)^2}.
\end{align}
The image of the moment map is the convex body defined by:
\begin{align}
\frac{\mu_2^{b\sqrt{-1}}}{\sqrt{-1}} & \leq 0, &
\frac{\mu_2^{b\sqrt{-1}}}{\sqrt{-1}} + \frac{\mu_1^{b\sqrt{-1}}}{\sqrt{-1}} & \geq - \frac{b}{2}, &
\frac{\mu_2^{b\sqrt{-1}}}{\sqrt{-1}} + 2\frac{\mu_1^{b\sqrt{-1}}}{\sqrt{-1}} & \geq 0.
\end{align}

We can now make a comparison of all the cases in this subsection and make a summary.
One can introduce a parameter $T = c^2$ which plays the role of the temperature,
and consider the family $\psi^\vee_{\sqrt{T}}$, $\omega_{\sqrt{T}}$
and $(\mu_1^{\sqrt{T}}, \mu_2^{\sqrt{T}})$ defined on
$X_2=\cO_{\bP^1}(-2)$.
A close examination of the computations in this subsection shows that
one can use the same formulas for both the $T\geq 0$ and the $T< 0$ cases.
When $T$ changes from a positive number to a negative number,
a phase transition happens at $T = 0$.
For $T > 0$, $\omega_{\sqrt{T}}$ is defined on the whole space of the line bundle
$\cO_{\bP^1}(-2)$;
when $T=0$,
the metric blows up along the zero section of the line bundle;
when $T< 0$,
$\omega_{\sqrt{T}}$ defines a K\"ahler metric
outside a circle bundle of $\cO_{\bP^1}(-2)$,
and it defines an imaginary pseudo-K\"ahler metric inside it.

The appearance of imaginary pseudo-K\"ahler metrics suggests that one can consider
$\psi^\vee_c$, $\omega_c$ and $(\mu_1^c, \mu_2^c)$ defined for all $c = a + bi \in \bC$
using the same formulas in the beginning of this subsection.
This means in general we consider complex-valued K\"ahler form
with nonvanishing real and imaginary parts.

\subsection{Phase changes of toric Gibbons-Hawking metrics}

By the discussions in \S \ref{sec:General}
we know that for $\vec{c} = (c_1, \dots, c_n)$ with pairwise distinct components,
one can define
on $X_n = \widehat{\bC^2/\bZ_n}$ a K\"ahler metric $\omega_{\vec{c}}$
with a K\"ahler potential $\psi^\vee_{\vec{c}}$,
and a Hamiltonian $T^2$-action with moment map $\mu^{\vec{c}}=(\mu^{\vec{c}}_1, \mu_2^{\vec{c}})$.
By allowing $\vec{c}$ to run in $\vec{\bC}^n$,
one can then obtain many possibilities of phase changes
of the toric Gibbons-Hawking metrics.
In extending the formulas for $\vec{c} \in \bR^n$ to $\vec{c} \in \bC^n$
one may encounter the problem of taking different branches
hence may have multi-valued solutions and may have to deal with
interesting monodromy problems.
We will leave these issues to future investigations.

\section{Concluding Remarks}

\label{sec:Conclusions}

The method of moment maps and Hessian geometry have been developed
in the mathematical literature in the context of toric canonical K\"ahler  metrics
on compact toric manifolds \cite{Gui, Abr, Don},
and they have been developed in the physics literature
in the context of AdS/CFT correspondence involving
toric Sasaki-Einstein metrics and noncompact Calabi-Yau metric cones \cite{Mar-Spa-Yau}.
More recent developments have generalized to K\"ahler metrics on
noncompact toric spaces obtained by explicit constructions,
and have been applied to the Kepler problem \cite{Zhou-Kepler}.
Earlier results on convexity of moment maps and Hessian geometry
in the noncompact case have focused on the case of convex cones.
In these new developments many examples having
convex bodies as moment images have been found.
Furthermore,
new approach to the phase phenomena introduced in earlier work \cite{Duan-Zhou1, Duan-Zhou2}
based on such convexity  has been developed \cite{Wang-Zhou}.

This work is the result of a natural continuation of the ideas in \cite{Zhou-Kepler}.
The theme is still to establish a link
between gravity and  string theory,
guided by the basic principle of applying the intrinsic symmetry of the geometry.
Therefore,
we have focused on the toric Gibbons-Hawking metrics
to have larger symmetry groups.

The phase change we study in this paper leads us to imaginary K\"ahler metrics.
This is actually not so surprising from the point of view of either general relativity or
string theory.
In general relativity,
it is a common practice to change from a Lorentzian space-time to a Euclidean space-time
by Wick rotation,
i.e.,
changing to purely imaginary time.
This complexification of the time variable led Penrose
to consider the complexification of the whole space-time and metric.
The physical spact-time is then often a real slice of this complexified space-time.
Our work suggests the possibility of the following scenario when one takes
a slice of the complexified space-time:
At certain part of the space-time the metric is real-valued,
but at certain part of the space-time the metric is purely imaginary
and such a part is regarded as a blackhole.
Of course,
one can also encounter a slice along which the metric becomes complex-valued,
i.e., the real part and the imaginary part of it can be both nonvanishing.
But this is very natural from the point of view nonlinear sigma models 
or gauged linear sigma models in string theory \cite{Hori}.
When the real part  is K\"ahler,
one gets a complexified K\"ahler form $\omega - i B$ in the string theory literature.

In the example of Eguchi-Hanson space,
we introduce a parameter $c=a+b\sqrt{-1}$ to induce the phase change.
We regard this parameter as a ``complexified temperature".
Another motivation behind the work \cite{Zhou-Kepler} and this paper
is to search for a statistical  reformulation of gravity
by maximum entropy principle,
partly supported by the appearance of expressions like $p\log p$ in many places
in these two papers.
The discussion of phase change in this paper seems to suggest that
suitable ``complexification of entropy", and more generally,
``complexification of statistical mechanics", should be useful for such purpose.

\vspace{.2in}
{\em Acknowledgements}.
The research in this work is partially supported by NSFC grant 11661131005.

\end{document}